\documentclass[11pt]{amsart}

\usepackage{amssymb}
\usepackage{mathrsfs}
\usepackage{amsfonts}
\usepackage{latexsym,amsmath,amsthm,amssymb,amsfonts}
\usepackage[usenames]{color}
\usepackage{xspace,colortbl}
\usepackage{graphicx}
\usepackage{tipa}

\newcommand{\be}{\begin{equation}}
\newcommand{\ee}{\end{equation}}
\newcommand{\beq}{\begin{eqnarray}}
\newcommand{\eeq}{\end{eqnarray}}

\newtheorem{prop}{Proposition}[section]

\newtheorem{remark}[prop]{Remark}

\def\begeq{\begin{equation}}
\def\endeq{\end{equation}}

\def\odot{\setbox0=\hbox{$\bigcirc$}\relax \mathbin {\hbox
to0pt{\raise.5pt\hbox to\wd0{\hfil $\wedge$\hfil}\hss}\box0 }}

\numberwithin{equation} {section}

\numberwithin{equation}{section}
\newtheorem{theorem}{\bf Theorem}[section]

\newtheorem{lemma}[theorem]{\bf Lemma}

\newtheorem{corollary}[theorem]{\bf Corollary}

\allowdisplaybreaks

\begin{document}
\title[Asymptotic convergence for a class of fully nonlinear ICFs in a cone]
 {Asymptotic convergence for a class of fully nonlinear inverse
curvature flows in a cone}

\author{ Ya Gao$^{\dagger}$,~~~Jing Mao$^{\ddagger,\ast}$
 }

\address{
 $^{\dagger}$School of Mathematical Science and Academy for Multidisciplinary Studies, Capital Normal University, Beijing 100048, China. }

\email{Echo-gaoya@outlook.com}

\address{
 $^{\ddagger}$Faculty of Mathematics and Statistics, Key Laboratory of
Applied Mathematics of Hubei Province, Hubei University, Wuhan
430062, China. }

\email{jiner120@163.com}

\thanks{$\ast$ Corresponding author}

\date{}
\begin{abstract}
For a given smooth convex cone in the Euclidean $(n+1)$-space
$\mathbb{R}^{n+1}$ which is centered at the origin, we investigate
the evolution of strictly mean convex hypersurfaces, which are
star-shaped with respect to the center of the cone and which meet
the cone perpendicularly, along an inverse curvature flow with the
speed equal to $\left(f(r)H\right)^{-1}$, where $f$ is a positive
function of the radial distance parameter $r$ and $H$ is the mean
curvature of the evolving hypersurfaces.  The evolution of those
hypersurfaces inside the cone yields a fully nonlinear parabolic
Neumann problem. Under suitable constraints on the first and the
second derivatives of the radial function $f$, we can prove the
long-time existence of this flow, and moreover the evolving
hypersurfaces converge smoothly to a piece of the round sphere.
\end{abstract}

\maketitle {\it \small{{\bf Keywords}: Inverse curvature flow,
Neumann boundary condition, asymptotic convergence.}

}

\section{Introduction}

Let $\mathscr{M}_{0}$ be a smooth, closed and convex hypersurface in
the Euclidean $(n+1)$-space $\mathbb{R}^{n+1}$ which encloses the
origin, $n\geq1$. Li, Sheng and Wang \cite{LSW} investigated the
evolution of  $\mathscr{M}_{0}$ along the geometric flow
\begin{equation}\label{flow1}
\left\{
\begin{aligned}
&\frac{\partial X}{\partial t}(x,t)=-r^{\alpha}\sigma_{k}(x,t)\nu, \\
& X(x,0)=X_{0}(x),\\
\end{aligned}
\right.
\end{equation}
where $\sigma_{k}$ is the $k$-th Weingarten curvature (i.e., the
$k$-th elementary symmetric function of principal curvatures) given
by
\begin{eqnarray*}
 \sigma_{k}(\cdot,t)=\sum\limits_{1\leq i_{1}<i_{2}<\cdots<i_{k}\leq
 n}\lambda_{i_1}\lambda_{i_2}\cdots\lambda_{i_k},
\end{eqnarray*}
with $\lambda_{i}=\lambda_{i}(\cdot,t)$ the principal curvatures of
the hypersurface $\mathscr{M}_{t}$, parameterized by\footnote{~In
this paper, $\mathbb{S}^{n}$ stands for the unit Euclidean
$n$-sphere.} $X(\cdot,t):\mathbb{S}^{n}\rightarrow\mathbb{R}^{n+1}$.
Besides, in the system (\ref{flow1}), $\nu(\cdot,t)$ is the unit
outward  normal vector of $\mathscr{M}_{t}$ at $X(\cdot,t)$,
$\alpha\in\mathbb{R}^{1}$ and $r:=|X(\cdot,t)|$ denotes the radial
distance from the point $X(x,t)$ to the origin. It is not hard to
see that: if $\alpha=0$, $k=1$, the flow (\ref{flow1}) degenerates
into the well-known mean curvature flow;  $\alpha=0$, $k=n$, the
flow (\ref{flow1}) becomes the Gauss curvature flow. The mean
curvature flow and the Gauss curvature flow have been extensively
studied in the past six decades, and some interesting, important
results have been obtained -- see e.g. \cite{BA,KSC,WJF,GH}. BTW, if
$\alpha=0$, then in this setting the evolution equation of
(\ref{flow1}) is simply called $\sigma_{k}$-flow equation and the
speed of the $\sigma_{k}$-flow is just the $k$-th Weingarten
curvature (sometimes, $k$-curvature for abbreviation). If $\alpha$
does not vanish, then the speed of the flow (\ref{flow1}) depends
not only on the curvatures but also on the radial function. As shown
by Li, Sheng and Wang in \cite{LSW}, the motivations of their study
on the flow (\ref{flow1}) include not only the close relation with
those classical curvature flows just mentioned above but also its
background in convex geometry (see \cite[pp. 835-836]{LSW} for an
interesting explanation). For the flow (\ref{flow1}), Li, Sheng and
Wang \cite{LSW} proved that:

\begin{itemize}
\item \emph{If $\alpha\geq k+1$, the flow exists for all time, preserves
the convexity and converges smoothly after normalisation to a sphere
centred at the origin. If $\alpha<k+1$, a counterexample (see
\cite[Sect. 5]{LSW}) is given for the above convergence. In the case
$k=1$ and $\alpha\geq2$, the flow converges to a round point if the
initial hypersurface is weakly mean-convex and star-shaped.}
\end{itemize}
Inspired by the work \cite{LSW}, Mao and his collaborators
\cite{CMTW} considered the evolution of closed, star-shaped,
strictly mean convex $C^{2,\gamma}$ hypersurfaces in
$\mathbb{R}^{n+1}$ along the unit outward  normal vector with a
speed equal to $\left(r^{\alpha}H\right)^{-1}$, where
$H=\sigma_{1}(\cdot,t)$ is the mean curvature of the hypersurfaces,
and proved that if $\alpha\geq0$, this evolution exists for all time
and the evolving hypersurfaces converge smoothly to a round sphere
after rescaling. Clearly, when $\alpha=0$, the expanding flow
considered in \cite{CMTW} degenerates into the well-known inverse
mean curvature flow (IMCF) and then asymptotic convergence
conclusion in \cite{CMTW} certainly covers the convergent result of
the classical IMCF firstly obtained by Gerhardt \cite{Ge90} (or
Urbas \cite{Ur}).

The curvature flows mentioned in this paper so far actually
correspond to the evolution of closed hypersurfaces along some
prescribed curvature flow. What about the situation that the initial
hypersurface has boundary? This implies that when considering some
prescribed curvature flow, we have not only the initial value
condition (IVC) as usual but also the boundary value condition
(BVC). This addition of BVC would definitely increase the difficulty
of getting a priori estimates for the solutions to the prescribed
flow equation.

To the best of our knowledge, the first interesting convergent
result concerning the IMCF with boundary is due to T. Marquardt
\cite{Mar1,Mar}. More precisely, let $M^{n}\subset\mathbb{S}^n$
($n\geq2$) be some piece of $\mathbb{S}^n$ such that
$\Sigma^{n}:=\{rx\in \mathbb{R}^{n+1}| r>0, x\in
\partial M^n\}$ be the boundary of a smooth, convex cone that is centered at
the origin. Marquardt \cite{Mar1,Mar} investigated the evolution of
hypersurfaces with boundary (which are star-shaped with respect to
the center of the cone $\Sigma^{n}$ and which also meet $\Sigma^{n}$
perpendicular) along the IMCF, and they can prove that the long-time
existence of this flow and moreover the evolving hypersurfaces
converge smoothly and exponentially to a piece of the round sphere.
Inspired by the previous work \cite{CMTW}, together with the success
of curvature estimates on the boundary, Mao and Tu \cite{mt}
generalized Marquardt's above conclusion from the IMCF to the
inverse curvature flow with the speed equal to
$\left(r^{\alpha}H\right)^{-1}$, $\alpha\geq0$. Is it possible to
consider some inverse curvature flow with boundary in a more general
geometric space (not the Euclidean space)? The answer is affirmative
and we have obtained several successful cases -- see e.g. \cite{GM1}
for the evolution of spacelike (strictly mean convex) graphic
hypersurfaces along the IMCF with boundary in a time cone of the
Lorentz-Minkowski space $\mathbb{R}^{n+1}_{1}$, \cite{GM} for the
inverse curvature flow with boundary and speed equal to
$\left(r^{\alpha}H\right)^{-1}$ in a time cone of
$\mathbb{R}^{n+1}_{1}$, \cite{GM2} for the inverse Gauss curvature
flow in a time cone of $\mathbb{R}^{n+1}_{1}$.

Very recently, inspired by the work \cite{CMTW} of Mao and his
collaborators, Gong \cite{AHG} found that similar convergent result
can be obtained if the flow speed  $\left(r^{\alpha}H\right)^{-1}$,
$\alpha\geq0$ was replaced by a more general one
$\left(f(r)H\right)^{-1}$, where $f$ is a positive function of the
radial distance parameter $r=|X(\cdot,t)|$ and also satisfies two
constraints involving the first and the second derivatives of $f$
(see \cite[p. 642]{AHG} for details).

It is natural to ask whether or not a similar story could be
expected in the case that the initial hypersurface also has a
boundary when considering the IMCF. The purpose of this paper is
trying to give an affirmative answer to this question. In fact, by
successfully overcoming difficulties in curvature estimates on the
boundary, we can prove the following:

\begin{theorem}\label{main1.1}
Let $M^n\subset\mathbb{S}^{n}$ be some convex piece of the unit
sphere $\mathbb{S}^{n}\subset\mathbb{R}^{n+1}$, and
$\Sigma^n=\{rx\in \mathbb{R}^{n+1}| r>0, x\in
\partial M^n\}$ be the boundary of a smooth, convex cone that is centered at the origin and has outward unit normal $\mu$. Let
$X_{0}:M^{n}\rightarrow\mathbb{R}^{n+1}$ such that
$M_{0}^{n}:=X_{0}(M^{n})$ is a compact, strictly mean convex $C^{2,\gamma}$-hypersurface ($0<\gamma<1$) which is star-shaped with respect to the center of the cone.
 Assume that
 \begin{eqnarray*}
M_{0}^{n}=\mathrm{graph}_{M^n}u_{0}
 \end{eqnarray*}
 is a graph over $M^n$ for a positive
map $u_0: M^n\rightarrow \mathbb{R}$ and
 \begin{eqnarray*}
\partial M_{0}^{n}\subset \Sigma^n, \qquad
\langle\mu\circ X_{0}, \nu_0\circ X_{0} \rangle|_{\partial
M^n}=0,
 \end{eqnarray*}
 where $\nu_0$ is the unit normal vector of $M_{0}^{n}$.
Then we have:

(i) There exists a family of strictly mean convex hypersurfaces $M_{t}^{n}$ given by the unique embedding
\begin{eqnarray*}
X\in C^{2+\gamma,1+\frac{\gamma}{2}} (M^n\times [0,\infty),
\mathbb{R}^{n+1}) \cap C^{\infty} (M^n\times (0,\infty),
\mathbb{R}^{n+1})
\end{eqnarray*}
with $X(\partial M^n, t) \subset \Sigma^n$ for $t\geq 0$, satisfying the following system
\begin{equation}\label{Eq}
\left\{
\begin{aligned}
&\frac{\partial }{\partial t}X=\frac{1}{f(|X|)H}\nu ~~&&in~
M^n \times(0,\infty)\\
&\langle \mu \circ X, \nu\circ X\rangle=0~~&&on~ \partial M^n \times(0,\infty)\\
&X(\cdot,0)=M_{0}^{n}  ~~&& in~M^n
\end{aligned}
\right.
\end{equation}
where $H$ is the mean curvature of
$M_{t}^{n}:=X(M^n,t)=X_{t}(M^{n})$, $\nu$ is the unit outward normal vector of $M_{t}^{n}$, and the function $f(y)\in C^{0}([0,\infty))\cap C^{2}((0,\infty))$ satisfies: \\
(1) $f(y)\geq 0$, $f(y)=0$ if and only if $y=0$;\\
(2) $0<c_{1}\leq yf'(y)f^{-1}(y)\leq c_{2}$ for some constants $c_{1}$ and $c_{2}$, or $f$ is a constant function$\footnote{\rm{~In $\dot{\varphi}$ estimate, we need the condition that $0<c_1\leq yf'(y)f^{-1}(y)\leq c_2$ in order to apply the maximum principle. But when $f$ is a constant function, the third term on the right hand side of the first equation of \eqref{3.4} would not appear any more. So here we give the assumption for two different cases of $f$.}}$;\\
(3) If there exist some positive constants $c_3$, $c_4$ such that
$c_3\leq f(y)\leq c_4$, then exist some positive constant $c_5$,
$c_6$ such that $c_5f(g(t))\leq f(y\cdot g(t))\leq c_6f(g(t))$,
where the function $g(t)$ just depends on the time variable $t$.

Moreover, the H\"{o}lder norm on
the parabolic space $M^n\times(0,\infty)$ is defined in the usual
way (see, e.g., \cite[Note 2.5.4]{Ge3}).

(ii) The leaves $M_{t}^{n}$ are graphs over $M^n$, i.e.,
 \begin{eqnarray*}
M_{t}^{n}=\mathrm{graph}_{M^n}u(\cdot, t).
\end{eqnarray*}

(iii) Moreover, the evolving hypersurfaces converge
smoothly after rescaling to a piece of a round sphere of radius $r_\infty$ , where $r_{\infty}$ satisfies
\begin{eqnarray*}
    \frac{1}{\sup\limits_{M^{n}}u_{0}}\left(\frac{\mathcal{H}^n(M_{0}^{n})}{\mathcal{H}^n(M^{n})}\right)^{\frac{1}{n}}\leq
    r_{\infty}
    \leq\frac{1}{\inf\limits_{M^{n}}u_{0}}\left(\frac{\mathcal{H}^n(M_{0}^{n})}{\mathcal{H}^n(M^{n})}\right)^{\frac{1}{n}},
\end{eqnarray*}
where $\mathcal{H}^n(\cdot)$ stands for the $n$-dimensional
Hausdorff measure of a prescribed Riemannian $n$-manifold.
\end{theorem}

In Remark \ref{remark 1}, Remark \ref{remark 2} and Remark
\ref{remark 3} below, we would give the detailed explanations about
why the function $f$ needs to satisfy three assumptions (1)-(3) in
Theorem \ref{main1.1}.

\begin{remark}
\rm{ It is not hard to construct the following possible examples
which do satisfy the assumptions (1)-(3) in Theorem \ref{main1.1}:\\
$\bullet$ ~~$f(y)=y^{\alpha}$, $\alpha\geq0$; \\
$\bullet$ ~~$f(y)=y+\log (1+y)$; \\
$\bullet$ ~~$f(y)=\frac{e^y \cdot y}{1+e^y}$. }

\end{remark}

\begin{remark}
\rm{It is not hard to find that as long as the function $f$
satisfies three conditions in Theorem \ref{main1.1}, the evolving
hypersurfaces will converge after rescaling to a piece of a round
sphere of certain radius $r_\infty$, which can be controlled by two
positive constants that only related to the initial hypersurface. }

\end{remark}

This paper is organized as follows. In Section \ref{se3}, we will
show that the flow equation (which generally is a system of PDEs)
changes into a single scalar second-order parabolic PDE. In Section
\ref{se4}, several estimates, including $C^0$, time-derivative and
gradient estimates, of solutions to the flow equation will be shown
in details. Estimates of higher-order derivatives of solutions to
the flow equation, which naturally leads to the long-time existence
of the flow, will be investigated in Section \ref{se5}. We will
clearly show the convergence of the rescaled flow in Section
\ref{se6}.

\section{The scalar version of the flow equation} \label{se3}

Since the initial $C^{2,\gamma}$-hypersurface is star-shaped,
there exists
a function $u_0\in C^{2,\gamma} (M^{n})$ such that
 $X_0: M^{n} \rightarrow \mathbb{R}^{n+1}$ has the form $x \mapsto
G_{0}:=(x,u_0(x))$. The hypersurface $M_{t}^{n}$ given by the
embedding
\begin{eqnarray*}
X(\cdot, t): M^{n}\rightarrow \mathbb{R}^{n+1}
\end{eqnarray*}
at time $t$ may be represented as a graph over $M^n\subset
\mathbb{S}^{n}$, and then we can make ansatz
\begin{eqnarray*}
X(x,t)=\left(x,u(x,t)\right)
\end{eqnarray*}
for some function $u: M^{n} \times [0,T) \rightarrow \mathbb{R}$.

Using techniques as in Ecker \cite{Eck} (see also \cite{Ge90, Ge3,
Mar}), the problem \eqref{Eq} is degenerated into solving the
following scalar equation with the corresponding initial data and
the corresponding NBC
\begin{equation}\label{Eq-}
\left\{
\begin{aligned}
&\frac{\partial u}{\partial t}=\frac{v}{f(|X|)H} \qquad &&~\mathrm{in}~
M^n\times(0,\infty)\\
&D_{\mu} u=0  \qquad&&~\mathrm{on}~ \partial M^n\times(0,\infty)\\
&u(\cdot,0)=u_{0} \qquad &&~\mathrm{in}~M^n.
\end{aligned}
\right.
\end{equation}
By \cite[Lemma 2.1]{GM}, define a new function $\varphi(x,t)=\log
u(x, t)$ and then the mean curvature can be rewritten as
\begin{eqnarray*}
H=\sum_{i=1}^{n}h^i_i=\frac{e^{-\varphi}}{
v}\bigg(n-(\sigma^{ij}-\frac{\varphi^{i}\varphi^{j}}{v^{2}})\varphi_{ij}\bigg).
\end{eqnarray*}
Hence, the evolution equation in \eqref{Eq-} can be rewritten as
\begin{eqnarray*}
\frac{\partial}{\partial t}\varphi=f^{-1}(e^{\varphi}) (1+|D\varphi|^2)\frac{1}
{[n-(\sigma^{ij}-\frac{\varphi^{i}
\varphi^{j}}{v^2})\varphi_{ij}]}:=Q(\varphi,D\varphi, D^2\varphi).
\end{eqnarray*}
In particular,
 \begin{eqnarray*}
\left(n-(\sigma^{ij}-\frac{\varphi_0^{i}
\varphi_0^{j}}{v^2})\varphi_{0,ij}\right)
 \end{eqnarray*}
 is positive on $M^n$, since $M_0$ is strictly mean convex.
Thus, the problem \eqref{Eq} is again reduced to solve the following
scalar equation with the NBC and the initial data
\begin{equation}\label{Evo-1}
\left\{
\begin{aligned}
&\frac{\partial \varphi}{\partial t}=Q(\varphi,D\varphi, D^{2}\varphi) \quad
&& \mathrm{in} ~M^n\times(0,T)\\
&D_{\mu} \varphi =0  \quad && \mathrm{on} ~ \partial M^n\times(0,T)\\
&\varphi(\cdot,0)=\varphi_{0} \quad && \mathrm{in} ~ M^n,
\end{aligned}
\right.
\end{equation}
where
$$\left(n-(\sigma^{ij}-\frac{\varphi_0^{i}
\varphi_0^{j}}{v^2})\varphi_{0,ij}\right)$$ is positive on $M^n$.
Clearly, for the initial graphic hypersurface $M_{0}^{n}$,
$$\frac{\partial Q}{\partial \varphi_{ij}}\Big{|}_{\varphi_0}=\frac{1}{e^{2\varphi_{0}} f(e^{\varphi_{0}})H^{2}}(\sigma^{ij}-\frac{\varphi_0^{i}
\varphi_0^{j}}{v^2})$$ is positive on $M^n$, since $f^{-1}(e^{\varphi_{0}})>0$. Based on the above
facts, as in \cite{Ge90, Ge3, Mar}, we can get the following
short-time existence and uniqueness for the parabolic system
\eqref{Eq}.

\begin{lemma}
Let $X_0(M^n)=M_{0}^{n}$ be as in Theorem \ref{main1.1}. Then there
exist some $T>0$, a unique solution  $u \in
C^{2+\gamma,1+\frac{\gamma}{2}}(M^n\times [0,T]) \cap C^{\infty}(M^n
\times (0,T])$, where $\varphi(x,t)=\log u(x,t)$, to the parabolic
system \eqref{Evo-1} with the matrix
 \begin{eqnarray*}
\left(n-(\sigma^{ij}-\frac{\varphi^{i}
\varphi^{j}}{v^2})\varphi_{ij}\right)
 \end{eqnarray*}
positive on $M^n$. Thus there exists a unique map $\psi:
M^n\times[0,T]\rightarrow M^n$ such that $\psi(\partial M^n
,t)=\partial M^n$ and the map $\widehat{X}$ defined by
\begin{eqnarray*}
\widehat{X}: M^n\times[0,T)\rightarrow \mathbb{R}^{n+1}:
(x,t)\mapsto X(\psi(x,t),t)
\end{eqnarray*}
has the same regularity as stated in Theorem \ref{main1.1} and is
the unique solution to the parabolic system \eqref{Eq}.
\end{lemma}

Let $T^{\ast}$ be the maximal time such that there exists some
 \begin{eqnarray*}
u\in C^{2+\gamma,1+\frac{\gamma}{2}}(M^n\times[0,T^{\ast}))\cap
C^{\infty}(M^n\times(0,T^{\ast}))
 \end{eqnarray*}
  which solves \eqref{Evo-1}. In the
sequel, we shall prove a priori estimates for those admissible
solutions on $[0,T]$ where $T<T^{\ast}$.

\section{$C^0$, $\dot{\varphi}$ and gradient estimates} \label{se4}

\begin{lemma}[\bf$C^0$ estimate]\label{lemma3.1}
Let $\varphi$ be a solution of \eqref{Evo-1}, we have
\begin{equation*}
c_7 \leq u(x, t) \Theta^{-1}(t, c) \leq c_8, \qquad\quad \forall~
x\in M^n, \ t\in[0,T]
\end{equation*}
for some positive constants $c_7$, $c_8$, where $\Theta(t, c):=e^{\mathcal{F}(t)+c}$ with
$$\inf_{M^{n}}\varphi(\cdot,0)\leq c\leq \sup_{M^{n}}\varphi(\cdot,0),$$
where the definition of $\mathcal{F}(t)$ will be given later.
\end{lemma}

\begin{proof}
Let $\varphi(x, t)=\varphi(t)$ (independent of $x$) be  the solution
of \eqref{Evo-1} with $\varphi(0)=c$. In this case, the first
equation in \eqref{Evo-1} reduces to an ODE
\begin{eqnarray*}
\frac{d}{d t}\varphi(t)=\frac{1}{nf(e^{\varphi(t)})}>0.
\end{eqnarray*}
So we also have
\begin{eqnarray*}
  \frac{dt}{d \varphi}=nf(e^{\varphi(t)})>0.
  \end{eqnarray*}
Therefore, using the inverse function theorem, there exists the
function $\mathcal{F}(t)$, and $\mathcal{F}'(t)>0$. One might
further request $\mathcal{F}(0) = 0$, and so
\begin{eqnarray}\label{blow}
  \varphi(t) = \mathcal{F}(t) + c.
\end{eqnarray}
Using the maximum principle, we can obtain that
\begin{equation}\label{C^0}
\mathcal{F}(t)+\varphi_{1}\leq\varphi(x, t)
\leq\mathcal{F}(t)+\varphi_{2},
\end{equation}
where $\varphi_1:=\inf_{M^n}\varphi(\cdot,0)$ and
$\varphi_2:=\sup_{M^n} \varphi(\cdot,0)$. The estimate is obtained
since $\varphi =\log u$.
\end{proof}

\begin{remark}\label{remark 1}
\rm{Comparing with \cite[Lemma 3.1]{GM}, here since
$e^{\varphi(t)}\neq 0$, we have $f(e^{\varphi(t)})>0$ (the first
assumption in Theorem \ref{main1.1}). Then using the existence
theorem of inverse functions, we know that there exists a function
$\mathcal{F}(t)$ such that $\varphi(t) = \mathcal{F}(t) + c$ without
the need of explicit expression. }
\end{remark}

\begin{lemma}[\bf$\dot{\varphi}$ estimate]\label{lemma3.2}
Let $\varphi$ be a solution of \eqref{Evo-1} and $\Sigma^n$ be the
boundary of a smooth, convex domain defined as in Theorem
\ref{main1.1}, then
\begin{eqnarray*}
\min\left\{\inf_{M^n}(\dot{\varphi}(\cdot, 0)\cdot f(\Theta(0))), \frac{c_{1}}{nc_{2}}\right\} \leq \dot{\varphi}(x, t)f(\Theta(t))\leq
\max\left\{\sup_{M^{n}}(\dot{\varphi}(\cdot, 0)\cdot f(\Theta(0))), \frac{c_{1}}{nc_{2}}\right\}.
\end{eqnarray*}
\end{lemma}

\begin{proof}
Set
\begin{eqnarray*}
\mathcal{M}(x,t)=\dot{\varphi}(x, t)\cdot f(\Theta).
\end{eqnarray*}
Differentiating both sides of the first evolution equation of
\eqref{Evo-1}, it is easy to get that
\begin{equation} \label{3.4}
\left\{
\begin{aligned}
&\frac{\partial\mathcal{M}}{\partial t}=
Q^{ij}D_{ij}\mathcal{M}+Q^{k}D_k \mathcal{M} \\
&\qquad\qquad
+ f^{-1}(\Theta)\left[\frac{\Theta}{nf(\Theta)}\cdot\frac{\partial f}{\partial\Theta} - \frac{e^{\varphi}}{f(e^{\varphi})}\cdot\frac{\partial f}{\partial e^{\varphi}}\cdot \mathcal{M}\right]\mathcal{M} \quad
&& \mathrm{in} ~M^n\times(0,T)\\
&D_{\mu}\mathcal{M}=0 \quad && \mathrm{on} ~\partial M^n\times(0,T)\\
&\mathcal{M}(\cdot,0)=\dot{\varphi}_0 \quad && \mathrm{on} ~ M^n,
\end{aligned}
\right.
\end{equation}
where $Q^{ij}:=\frac{ \partial Q}{\partial \varphi_{ij}}$
 and $Q^k:=\frac{ \partial Q}{\partial \varphi_{k}}$. Then use the second assumption of the function $f$ in the Theorem \ref{main1.1}, so we have
 \begin{equation*}
  c_1\leq \frac{e^{\varphi}}{f(e^{\varphi})}\cdot\frac{\partial f}{\partial e^{\varphi}}\leq c_2.
 \end{equation*}
Then the result follows from the maximum principle.
\end{proof}

\begin{remark}\label{remark 2}
\rm{In Lemma \ref{lemma3.2}, since we want to use the maximum
principle, we should need the function $f$ to satisfy the assumption
(2) in Theorem \ref{main1.1}. But when $f$ is a constant function,
the third term in the right hand side of the equation \eqref{3.4}
will disappear. Especially, if $f=1$, the equation \eqref{Eq} is the
classical inverse mean curvature flow.}
\end{remark}

\begin{lemma}[\bf Gradient estimate]\label{Gradient}
Let $\varphi$ be a solution of \eqref{Evo-1} and $\Sigma^n$ be the boundary of a smooth, convex domain described as in Theorem
\ref{main1.1}. Then we have,
\begin{equation}\label{Gra-est}
|D\varphi|\leq \sup_{M^n}|D\varphi(\cdot, 0)|, \qquad\quad
\forall~ x\in M^n, \ t\in[0,T].
\end{equation}
\end{lemma}

\begin{proof}
Set $\psi=\frac{|D \varphi|^2}{2}$. By differentiating  $\psi$, we
have
\begin{equation*}
\begin{aligned}
\frac{\partial \psi}{\partial t} =\frac{\partial}{\partial
t}\varphi_m \varphi^m = \dot{\varphi}_m\varphi^m =Q_m \varphi^m.
\end{aligned}
\end{equation*}
Then using the evolution equation of $\varphi$ in (\ref{Evo-1})
yields
\begin{eqnarray*}
\frac{\partial \psi}{\partial t}=Q^{ij}\varphi_{ijm} \varphi^m
+Q^k\varphi_{km} \varphi^m - e^{\varphi}f'(e^{\varphi})f^{-1}(e^{\varphi})Q|D\varphi|^{2}.
\end{eqnarray*}
Interchanging the covariant derivatives, we have
\begin{equation*}
\begin{aligned}
\psi_{ij}&=D_j(\varphi_{mi} \varphi^m)\\&=\varphi_{mij} \varphi^m+\varphi_{mi} \varphi^m_j\\
&=(\varphi_{ijm}+R^l_{jmi}\varphi_{l})\varphi^m+\varphi_{mi}\varphi^m_j.
\end{aligned}
\end{equation*}
Therefore, we can express $\varphi_{ijm} \varphi^m$ as
\begin{eqnarray*}
\varphi_{ijm} \varphi^m =\psi_{ij}-R^l_{jmi}\varphi_l
\varphi^m-\varphi_{mi} \varphi^m_j.
\end{eqnarray*}
Then, in view of the fact
$R_{ijml}=\sigma_{im}\sigma_{jl}-\sigma_{il}\sigma_{jm}$ on
$\mathbb{S}^{n}$, we have
\begin{equation}\label{gra}
\begin{aligned}
\frac{\partial \psi}{\partial t}&=Q^{ij}\psi_{ij}+Q^k \psi_k
-Q^{ij}(\sigma_{ij}|D\varphi|^2-\varphi_i
\varphi_j)\\&-Q^{ij}\varphi_{mi}
\varphi^{m}_{j} - e^{\varphi}f'(e^{\varphi})f^{-1}(e^{\varphi})Q|D\varphi|^{2}.
\end{aligned}
\end{equation}
Same discussion as in \cite[Lemma 3.3]{GM}, since $M^n$ is convex
and using the second assumption in Theorem \ref{main1.1},
$e^{\varphi}f'(e^{\varphi})f^{-1}(e^{\varphi})$ is positive, and so
we can get
\begin{equation*}
\left\{
\begin{aligned}
&\frac{\partial \psi}{\partial t}\leq Q^{ij}\psi_{ij}+Q^k\psi_k
\qquad &&\mathrm{in}~
M^n\times(0,T)\\
&D_{\mu} \psi \leq 0   && \mathrm{on}~\partial M^n\times(0,T)\\
&\psi(\cdot,0)=\frac{|D\varphi(\cdot,0)|^2}{2} \qquad
&&\mathrm{in}~M^n.
\end{aligned}\right .\end{equation*}
Using the maximum principle, we have
\begin{equation*}
|D\varphi|\leq \sup_{M^n}|D\varphi(\cdot, 0)|.
\end{equation*}
Our proof is finished.
\end{proof}

\begin{remark}
\rm{It is worth pointing out that the evolution hypersurfaces $M_t^n$ are always star-shaped under the assumption of Theorem \ref{main1.1}, since, by Lemma \ref{Gradient}, we have
$$\left\langle \frac{X}{|X|}, \nu \right\rangle = \frac{1}{v}$$
is bounded from below by a positive constant.}
\end{remark}

Combing the gradient estimate with $\dot{\varphi}$ estimate, we can
obtain
\begin{corollary}
If $\varphi$ satisfies \eqref{Evo-1}, then we have
\begin{eqnarray}\label{w-ij}
0<c_9\leq  H\Theta \leq c_{10}<+\infty,
\end{eqnarray}
where $c_9$ and $c_{10}$ are positive constants independent of
$\varphi$.
\end{corollary}

\section{H\"{o}lder Estimates and Convergence} \label{se5}

In this section, we define the  rescaled flow by
\begin{equation*}
\widetilde{X}=X\Theta^{-1}.
\end{equation*}
Thus,
\begin{equation*}
\widetilde{u}=u\Theta^{-1},
\end{equation*}
\begin{equation*}
\widetilde{\varphi}=\varphi-\log\Theta,
\end{equation*}
and the rescaled mean curvature is given by
\begin{equation*}
\widetilde{H}=H\Theta.
\end{equation*}
Then, the rescaled scalar curvature equation takes the form
\begin{equation*}
\frac{\partial}{\partial
t}\widetilde{u}=\frac{v}{f(u)\widetilde{H}}-\frac{1}{n}\widetilde{u}f^{-1}(\Theta).
\end{equation*}
Defining $t=t(s)$ by the relation
\begin{equation*}
\frac{dt}{ds} = f(\Theta)
\end{equation*}
such that $t(0) = 0$ and $t(S) = T$. Then $\widetilde{u}$ satisfies
\begin{equation}\label{Eq-re}
\left\{
\begin{aligned}
&\frac{\partial}{\partial
s}\widetilde{u}=\frac{vf(\Theta)}{f(\widetilde{u}\Theta)\widetilde{H}}-\frac{1}{n}\widetilde{u}
\qquad && \mathrm{in}~
M^n\times(0,S)\\
&D_{\mu} \widetilde{u}=0  \qquad && \mathrm{on}~ \partial M^n\times(0,S)\\
&\widetilde{u}(\cdot,0)=\widetilde{u}_{0}  \qquad &&
\mathrm{in}~M^n.
\end{aligned}
\right.
\end{equation}

\begin{lemma}\label{res-01}
Let $X$ be a solution of (\ref{Eq}) and $\widetilde{X}=X
\Theta^{-1}$ be the rescaled solution. Then
\begin{equation*}
\begin{split}
&D \widetilde{u}=D u \Theta^{-1}, ~~~~D \widetilde{\varphi}=D \varphi,~~~~ \frac{\partial \widetilde{u}}{\partial s}=\frac{ \partial u}{\partial t} \frac{f(\Theta)}{\Theta}- \frac{1}{n}u\Theta^{-1},\\
&\widetilde{g}_{ij}=
\Theta^{-2}g_{ij},~~~~\widetilde{g}^{ij}=\Theta^{2}
g^{ij},~~~~\widetilde{h}_{ij}=h_{ij}\Theta^{-1}.
\end{split}
\end{equation*}
\end{lemma}
\begin{proof}
These relations can be computed directly.
\end{proof}

\begin{lemma} \label{lemma4-3}
Let $u$ be a solution to the parabolic system \eqref{Evo-1}, where
$\varphi(x,t)=\log u(x,t)$, and $\Sigma^n$ be the boundary of a
smooth, convex domain described as in Theorem \ref{main1.1}. Then
there exist some $\beta>0$ and some $C>0$ such that the rescaled
function $\widetilde{u}(x,s):=u(x,t(s)) \Theta^{-1}(t(s))$ satisfies
\begin{equation}
 [D \widetilde{u}]_{\beta}+\left[\frac{\partial \widetilde{u}}{\partial s}\right]_{\beta}+[\widetilde{H}]_{\beta}\leq C(||u_{0}||_{C^{2+\gamma,1+\frac{\gamma}{2}}(M^n)}, n, \beta, M^n),
\end{equation}
where $[f]_{\beta}:=[f]_{x,\beta}+[f]_{s,\frac{\beta}{2}}$ is the
sum of the H\"{o}lder coefficients of $f$ in $M^n\times[0,S]$ with
respect to $x$ and $s$.
\end{lemma}

\begin{proof}
We divide our proof in three steps\footnote{~In the proof of Lemma
\ref{lemma4-3}, the constant $C$ may differ from each other.
However, we abuse the symbol $C$ for the purpose of convenience.}.

\textbf{Step 1:} We need to prove that
\begin{equation*}
  [D \widetilde{u}]_{x,\beta}+[D \widetilde{u}]_{s,\frac{\beta}{2}}\leq C(|| u_0||_{ C^{2+\gamma,1+\frac{\gamma}{2}}(M^n)}, n, \beta, M^n).
\end{equation*}
According to Lemmas \ref{lemma3.1}, \ref{lemma3.2} and
\ref{Gradient}, it follows that
$$|D \widetilde{u}|+\left|\frac{\partial \widetilde{u}}{\partial s}\right|\leq C(|| u_0||_{ C^{2+\gamma,1+\frac{\gamma}{2}}(M^n)}, M^n).$$
Fix $s$ and the equation (\ref{Evo-1}) can be rewritten
as an elliptic Neumann problem
 \begin{equation} \label{key1}
   \mbox{div}_{\sigma}\left(\frac{D \varphi}{\sqrt{1+|D\varphi|^2}}\right)=\frac{n}{\sqrt{1+|D\varphi|^2}}-f^{-1}(e^{\varphi})\cdot
    \frac{\sqrt{1+|D\varphi|^2}}{\dot{\varphi}}.
 \end{equation}
 Other discussions are the same as in \cite[Lemma 4.3]{GM}.

\textbf{Step 2:} The next thing to do is to show that
\begin{equation*}
  \left[\frac{\partial \widetilde{u}}{\partial s}\right]_{x,\beta}+\left[\frac{\partial \widetilde{u}}{\partial s}\right]_{s,\frac{\beta}{2}}\leq
  C(||u_0||_{ C^{2+\gamma,1+\frac{\gamma}{2}}(M^n)}, n, \beta,
  M^n).
\end{equation*}
  As $\frac{\partial}{\partial s}\widetilde{u}=\widetilde{u}\left(\frac{vf(\Theta)}{f(\widetilde{u}\Theta)\widetilde{u}\widetilde{H}}-\frac{1}{n}\right)$, it is enough to bound
  $\left[\frac{vf(\Theta)}{\widetilde{u} \widetilde{H}f(\widetilde{u}\Theta)}\right]_{\beta}$.
Set $\Omega(s):=\frac{vf(\Theta)}{\widetilde{u} \widetilde{H}f(\widetilde{u}\Theta)}$. Let $\widetilde{\nabla}$ be the Levi-Civita connection of
$\widetilde{M}_{s}:=\widetilde{X}(M^n,s)$ w.r.t. the metric
$\widetilde{g}$. To prove that $\Omega$ belongs to the De Giorgi class of functions, we will need the following evolution equation.

\begin{lemma}\label{EVQ}
Set $\Phi=\frac{1}{f(e^{\varphi})H} = \frac{1}{f(u)H} = \frac{1}{f(\widetilde{u}\Theta)H}$, $w=\langle X, \nu\rangle = \frac{u}{v}$ and $\Psi=
\frac{\Phi}{w} = \frac{v}{\widetilde{u}\widetilde{H}f(\widetilde{u}\Theta)}$. Under the assumptions of Theorem \ref{main1.1}, we have
\begin{equation}\label{div-for-1}
  \begin{aligned}
 \frac{\partial \Psi }{\partial t}&=\mathrm{div}_g (f^{-1}H^{-2} \nabla \Psi)-2H^{-2}f^{-1}(e^{\varphi})\Psi^{-1}|\nabla \Psi|^2 - f'\frac{u}{f}\Psi^{2} \\
 &\quad
 -f'f^{-1}\Psi^{2}\nabla^{i}u\langle X, X_{i}\rangle - f'f^{-2}H^{-2}\nabla_{i}u\nabla^{i}\Psi.
  \end{aligned}
\end{equation}
\end{lemma}

\begin{proof}
Here, for the continuity of the proof of Lemma \ref{lemma4-3}, we
put the proof of Lemma \ref{EVQ} in Appendix at the end of this
paper. In fact, the main method used for the proof of Lemma
\ref{EVQ} is similar to that in \cite{GM1, GM}, but we still give
all the details in Appendix to show the differences brought by the
usage of the more general function $f$.
\end{proof}

Combining (\ref{div-for-1}) and $\Omega(s) = \Psi f(\Theta)$, we get
\begin{equation}\label{div-form-02}
  \begin{split}
\frac{\partial \Omega}{\partial s}
&=\mbox{div}_{\widetilde{g}}\left( f(\Theta)f^{-1}(\widetilde{u}\Theta)\widetilde{H}^{-2} \widetilde{\nabla}
\Omega\right)-2 \widetilde{H}^{-2}f^{-1}(\widetilde{u}\Theta)f(\Theta) \Omega^{-1}
|\widetilde{\nabla} \Omega|^2_{\widetilde{g}} + \frac{1}{n}\frac{f'(\Theta)\Theta}{f(\Theta)}\cdot\Omega \\
&\quad
 - \frac{f'(u)u}{f(u)}\cdot\Omega^{2} - \frac{f'(u)u}{f(u)}\Omega^2 P - \frac{f'(u)u}{f^2(u)}\frac{f(\Theta)}{\widetilde{u}}\widetilde{H}^{-2}\widetilde{\nabla}_{i}\widetilde{u}\widetilde{\nabla}^{i}\Omega,
  \end{split}
\end{equation}
where $P:=u^{-1}\nabla^{i}u\langle X, X_{i}\rangle$. Applying Lemma \ref{Gradient}, we have
\begin{equation*}
|P|\leq |\nabla u|^{2}\leq C.
\end{equation*}
The weak formulation of (\ref{div-form-02}) is
\begin{equation}\label{div-form-03}
\begin{split}
\int_{s_0}^{s_1} \int_{\widetilde{M}_s}  &\frac{\partial
\Omega }{\partial s}  \eta d\mu_s ds  =\int_{s_0}^{s_1}
\int_{\widetilde{M}_s} \mbox{div}_{\widetilde{g}} (
f(\Theta)f^{-1}(\widetilde{u}\Theta)\widetilde{H}^{-2} \widetilde{\nabla} \Omega) \eta -2
\widetilde{H}^{-2}f^{-1}(\widetilde{u}\Theta)f(\Theta) \Omega^{-1} |\widetilde{\nabla}
\Omega|^2_{\widetilde{g}} \eta d\mu_s ds\\
&+
\int_{s_{0}}^{s_{1}}\int_{\widetilde{M}_{s}}\left( \frac{1}{n}\frac{f'(\Theta)\Theta}{f(\Theta)}\cdot\Omega - \frac{f'(u)u}{f(u)}\cdot\Omega^{2} - \frac{f'(u)u}{f(u)}\Omega^2 P - \frac{f'(u)u}{f^2(u)}\frac{f(\Theta)}{\widetilde{u}}\widetilde{H}^{-2}\widetilde{\nabla}_{i}\widetilde{u}\widetilde{\nabla}^{i}\Omega\right)\eta d\mu_s ds.
\end{split}
\end{equation}
Since $\nabla_{\mu} \widetilde{\varphi}=0$, the boundary integrals
all vanish, the interior and boundary estimates are basically the
same. We define the test function $\eta:=\xi^2 \Omega$, where
$\xi$ is a smooth function with values in $[0,1]$ and is supported
in a small parabolic neighborhood. Then
\begin{equation}\label{imcf-hec-for-02}
\begin{split}
\int_{s_0}^{s_1} \int_{\widetilde{M}_s}  \frac{\partial
\Omega }{\partial s}  \xi^2 \Omega d\mu_s ds=
\frac{1}{2}||\Omega \xi||_{2,\widetilde{M}_s}^2
\Big{|}_{s_0}^{s_1}-\int_{s_0}^{s_1} \int_{\widetilde{M}_s}  \xi
\dot{\xi} \Omega^2 d\mu_s ds.
\end{split}
\end{equation}
where $\dot{\xi}:=\frac{\partial\xi}{\partial s}$.
Using the divergence theorem and Young's inequality and the second assumption of the function $f$ in the Theorem \ref{main1.1}, we can obtain
\begin{equation}\label{imcf-hec-for-03}
\begin{split}
\int_{s_0}^{s_1} \int_{\widetilde{M}_s} & \mbox{div}_{\widetilde{g}} ( f(\Theta)f^{-1}(\widetilde{u}\Theta)\widetilde{H}^{-2} \widetilde{\nabla} \Omega)  \xi^2 \Omega  d\mu_sds\\
&=-\int_{s_0}^{s_1} \int_{\widetilde{M}_s} f(\Theta)f^{-1}(\widetilde{u}\Theta)  \widetilde{H}^{-2}
\xi^2\widetilde{\nabla}_i \Omega
\widetilde{\nabla}^i\Omega  d\mu_sds\\
&\qquad
-2\int_{s_0}^{s_1} \int_{\widetilde{M}_s}f(\Theta)f^{-1}(\widetilde{u}\Theta) \widetilde{H}^{-2} \xi \Omega\widetilde{\nabla}_i\Omega \widetilde{\nabla}^i \xi d\mu_sds\\
&\leq\int_{s_0}^{s_1} \int_{\widetilde{M}_s} c_{11} \widetilde{H}^{-2}
|\widetilde{\nabla} \xi|^2\Omega^2  d\mu_sds.
\end{split}
\end{equation}
and
\begin{equation}\label{5-9}
  \begin{aligned}
    &\int_{s_{0}}^{s_{1}}\int_{\widetilde{M}_{s}}\left( \frac{1}{n}\frac{f'(\Theta)\Theta}{f(\Theta)}\cdot\Omega - \frac{f'(u)u}{f(u)}\cdot\Omega^{2} - \frac{f'(u)u}{f(u)}\Omega^2 P \right.\\
    &\qquad\qquad\qquad\qquad
    \left. - \frac{f'(u)u}{f^2(u)}\frac{f(\Theta)}{\widetilde{u}}\widetilde{H}^{-2}\widetilde{\nabla}_{i}\widetilde{u}\widetilde{\nabla}^{i}\Omega\right)\xi^{2}\Omega d\mu_s ds\\
    &\leq
    C(c_{2},n,P)\int_{s_{0}}^{s_{1}}\int_{\widetilde{M}_{s}}\xi^{2}(\Omega^{2}+|\Omega|^{3})d\mu_{s}ds \\
    &\qquad\qquad
    + \int_{s_{0}}^{s_{1}}\int_{\widetilde{M}_{s}}\frac{f'(u)u}{f(u)}\cdot\frac{f(\Theta)}{f(\widetilde{u}\Theta)\cdot\widetilde{u}}\widetilde{H}^{-2}|\widetilde{\nabla}\widetilde{u}||\widetilde{\nabla}\Omega|\xi^{2}\Omega d\mu_{s}ds\\
    &\leq
    C(c_{2},n,P)\int_{s_{0}}^{s_{1}}\int_{\widetilde{M}_{s}}\xi^{2}(\Omega^{2}+|\Omega|^{3})d\mu_{s}ds + \frac{c_{2}}{2}\int_{s_{0}}^{s_{1}}\int_{\widetilde{M}_{s}}\widetilde{H}^{-2}\frac{f(\Theta)}{f(\widetilde{u}\Theta)}|\widetilde{\nabla}\Omega|^{2}\xi^{2}d\mu_{s}ds\\
    &\qquad\qquad
    +\frac{c_{2}}{2}\int_{s_{0}}^{s_{1}}\int_{\widetilde{M}_{s}}\widetilde{H}^{-2}\frac{f(\Theta)}{f(\widetilde{u}\Theta)\cdot\widetilde{u}^{2}}|\widetilde{\nabla}\widetilde{u}|^{2}\xi^{2}\Omega^{2}d\mu_{s}ds.
  \end{aligned}
\end{equation}

Combing (\ref{imcf-hec-for-02}), (\ref{imcf-hec-for-03}),
(\ref{5-9}) and the third assumption (3) of the function $f$ in
Theorem \ref{main1.1}, we have
 \begin{equation}\label{imcf-hec-for-04}
\begin{split}
&\frac{1}{2}||\Omega \xi||_{2,\widetilde{M}_s}^2
\Big{|}_{s_0}^{s_1}
+C(2-\frac{c_{2}}{2})\int_{s_0}^{s_1} \int_{\widetilde{M}_s} \widetilde{H}^{-2} |\widetilde{\nabla} \Omega|^2  \xi^2   d\mu_sds\\
&\qquad \leq \int_{s_0}^{s_1} \int_{\widetilde{M}_s}  \xi |\dot{\xi}|
\Omega^2 d\mu_sds +\int_{s_0}^{s_1} \int_{\widetilde{M}_s}c_{11}
\widetilde{H}^{-2} |\widetilde{\nabla} \xi|^2\Omega^2
d\mu_sds \\
&\qquad\qquad
+C(c_{2},n,P)\int_{s_{0}}^{s_{1}}\int_{\widetilde{M}_{s}}\xi^{2}(\Omega^{2}+|\Omega|^{3})d\mu_{s}ds\\
&\qquad\qquad
+\frac{c_{2}c_{5}}{2}\int_{s_{0}}^{s_{1}}\int_{\widetilde{M}_{s}}\widetilde{H}^{-2}\widetilde{u}^{-2}|\widetilde{\nabla}\widetilde{u}|^{2}\xi^{2}\Omega^{2}d\mu_{s}ds,
\end{split}
\end{equation}
 which implies that
 \begin{equation}\label{imcf-hec-for-06}
\begin{split}
&\frac{1}{2}||\Omega \xi||_{2,\widetilde{M}_s}^2
\Big{|}_{s_0}^{s_1}
+\frac{C(c_{2})}{\max(\widetilde{H}^{2}) }\int_{s_0}^{s_1} \int_{\widetilde{M}_s} |\widetilde{\nabla} \Omega|^2  \xi^2   d\mu_sds\\
& \leq \left(1+ \frac{c_{11}}{\min( \widetilde{H}^{2})}\right)
\int_{s_0}^{s_1} \int_{\widetilde{M}_s}  \Omega^2 (\xi
|\dot{\xi}| +|\widetilde{\nabla} \xi|^2)d\mu_s ds\\
& \qquad
+\left(C(c_{2},n,P)+\frac{c_{2}c_{5}\max{|\widetilde{\nabla}\widetilde{u}|^{2}}}{2\min{(\widetilde{H}^{2}\widetilde{u}^{2})}}\right)\int_{s_0}^{s_1} \int_{\widetilde{M}_s}\xi^{2}\Omega^{2}+\xi^{2}|\Omega|^{3}d\mu_{s}ds.
\end{split}
\end{equation}
This means that $\Omega$ belong to the De Giorgi class of
functions in $M^n \times [0,S)$. Similar to the arguments in
\cite[Chap. 5, \S 1 and \S 7]{La2},  there exist  constants $\beta$
and $C$ such that
$$[\Omega]_{\beta}\leq C ||\Omega||_{L^{\infty}(M^n \times [0,S))}\leq  C(|| u_0||_{ C^{2+\gamma,1+\frac{\gamma}{2}}(M^n)}, n, \beta, M^n).$$

\textbf{Step 3:} Finally, we have to show that
\begin{equation*}
  [ \widetilde{H}]_{x,\beta}+[\widetilde{H}]_{s,\frac{\beta}{2}}\leq C(|| u_0||_{ C^{2+\gamma,1+\frac{\gamma}{2}}(M^n)}, n, \beta, M^n).
\end{equation*}
This follows from the fact that
$$\widetilde{H}=\frac{\sqrt{1+|D\widetilde{\varphi}|^2}\cdot f(\Theta)}{\widetilde{u}f(\widetilde{u}\Theta) \Omega}$$
together with the estimates for $\widetilde{u}$, $\Omega$,
$D\widetilde{\varphi}$ and the conditions of $f$.
\end{proof}

\begin{remark}\label{remark 3}
\rm{In this paper, the function $f$ is much more general than that
in \cite{GM,mt} and we don't have the explicit expression. So we
only can write $f(\widetilde{u}\Theta)$ (related with
$\widetilde{u}$) in the rescaled scalar curvature equation rather
than $f(\widetilde{u})$. Furthermore, by Lemma \ref{lemma3.1} we
know $c_{7}\leq \widetilde{u}\leq c_{8}$ and $\Theta(t)$ is a
function only related to $t$. Using the third assumption (3) in
Theorem \ref{main1.1}, we  can also obtain the conclusion of Lemma
\ref{lemma4-3}.}
\end{remark}

Then we can obtain the following higher-order estimates:
\begin{lemma}
Let $u$ be a solution to the parabolic system \eqref{Evo-1}, where
$\varphi(x,t)=\log u(x,t)$, and $\Sigma^n$ be the boundary of a
smooth, convex domain described as in Theorem \ref{main1.1}. Then
for any $s_0\in (0,S)$ there exist some $\beta>0$ and some $C>0$
such that
\begin{equation}\label{imfcone-holder-01}
||\widetilde{u}||_{C^{2+\beta,1+\frac{\beta}{2}}(M^n\times
[0,S])}\leq C(|| u_0||_{ C^{2+\gamma, 1+\frac{\gamma}{2}}(M^n)}, n,
\beta, M^n)
\end{equation}
and for all $k\in \mathbb{N}$,
\begin{equation}\label{imfcone-holder-02}
||\widetilde{u}||_{C^{2k+\beta,k+\frac{\beta}{2}}(M^n\times
[s_0,S])}\leq C(||u_0(\cdot,
s_0)||_{C^{2k+\beta,k+\frac{\beta}{2}}(M^n)}, n, \beta, M^n).
\end{equation}
\end{lemma}

\begin{proof}
By \cite[Lemma 2.1]{GM}, we have
$$uvH=n-(\sigma^{ij}-\frac{\varphi^{i}\varphi^{j}}{v^{2}})\varphi_{ij}=n-u^2 \Delta_g \varphi.$$
Since
$$u^2 \Delta_g \varphi=\widetilde{u}^2 \Delta_{\widetilde{g}} \varphi=
-| \widetilde{\nabla} \widetilde{u}|^2+ \widetilde{u}
\Delta_{\widetilde{g}} \widetilde{u},$$ then
\begin{equation*}
\begin{split}
\frac{\partial \widetilde{u}}{\partial s}&=\frac{ \partial u}{\partial t} \Theta^{-1}f(\Theta)-\frac{1}{n} \widetilde{u}\\
&=-\frac{uvH}{u H^2f(e^{\varphi})} \Theta^{-1}f(\Theta) +\frac{2v}{Hf(e^{\varphi})} \Theta^{-1}f(\Theta)- \frac{1}{n}\widetilde{u}\\
&=\frac{\Delta_{\widetilde{g}} \widetilde{u}}{\widetilde{H}^2}\cdot\frac{f(\Theta)}{f(e^{\varphi})}+
\frac{2v}{\widetilde{H}}\cdot\frac{f(\Theta)}{f(e^{\varphi})} - \frac{1}{n}\widetilde{u} -\frac{n+|
\widetilde{\nabla} \widetilde{u}|^2}{\widetilde{u} \widetilde{H}^2}\cdot\frac{f(\Theta)}{f(e^{\varphi})},
\end{split}
\end{equation*}
which is  a uniformly parabolic equation with H\"{o}lder continuous
coefficients. Therefore, the linear theory (see \cite[Chap.
4]{Lieb}) yields the inequality (\ref{imfcone-holder-01}).

Set $\widetilde{\varphi}=\log \widetilde{u}$, and then the rescaled
version of the evolution equation in (\ref{Eq-re}) takes the form
\begin{equation*}
  \frac{\partial \widetilde{\varphi}}{\partial s}=\frac{f(\Theta)}{f(e^{\varphi})}\cdot\frac{ v^2}{ \left[n-\left(\sigma^{ij}-\frac{\widetilde{\varphi}^i\widetilde{\varphi}^j}{v^2}\right) \widetilde{\varphi}_{ij}\right]}-\frac{1}{n},
\end{equation*}
where $v=\sqrt{1+|D \widetilde{\varphi}|^2}$. According to the
$C^{2+\beta,1+\frac{\beta}{2}}$-estimate of $\widetilde{u}$ (see
Lemma \ref{lemma4-3}), we can treat the equations for $\frac{
\partial\widetilde{\varphi}}{\partial s}$ and $D_i
\widetilde{\varphi}$ as second-order linear uniformly parabolic PDEs
on $M^n\times [s_0,S]$. At the initial time $s_0$, all compatibility
conditions are satisfied and the initial function $u(\cdot,t_0)$ is
smooth. We can obtain a $C^{3+\beta, \frac{3+\beta}{2}}$-estimate
for $D_i \widetilde{\varphi}$ and a $C^{2+\beta,
\frac{2+\beta}{2}}$-estimate for $\frac{
\partial\widetilde{\varphi}}{\partial s}$ (the estimates are
independent of $T$) by Theorem 4.3 and Exercise 4.5 in \cite[Chap.
4]{Lieb}. Higher regularity can be proven by induction over $k$.
\end{proof}

\begin{theorem} \label{key-2}
Under the hypothesis of Theorem \ref{main1.1}, we conclude
\begin{equation*}
T^{*}=+\infty.
\end{equation*}
\end{theorem}
\begin{proof}
The proof of this result is quite similar to the corresponding
argument in \cite[Lemma 8]{Mar} and so is omitted.
\end{proof}

\section{Convergence of the rescaled flow} \label{se6}

We know that after the long-time existence of the flow has been
obtained (see Theorem \ref{key-2}), the rescaled version of the
system (\ref{Evo-1}) satisfies
\begin{equation}
\left\{
\begin{aligned}
&\frac{\partial}{\partial
s}\widetilde{\varphi}=\widetilde{Q}(\widetilde{\varphi},D\widetilde{\varphi},
D^2\widetilde{\varphi})  \qquad &&\mathrm{in}~
M^n\times(0,\infty)\\
&D_{\mu} \widetilde{\varphi}=0  \qquad &&\mathrm{on}~ \partial M^n\times(0,\infty)\\
&\widetilde{\varphi}(\cdot,0)=\widetilde{\varphi}_{0} \qquad
&&\mathrm{in}~M^n,
\end{aligned}
\right.
\end{equation}
where
$$\widetilde{Q}(\widetilde{\varphi}, D\widetilde{\varphi},
D^2\widetilde{\varphi}):=\frac{f(\Theta)}{f(e^{\widetilde{\varphi}}\cdot\Theta)}\cdot\frac{ v^2}{
\left[n-\left(\sigma^{ij}-\frac{\widetilde{\varphi}^i\widetilde{\varphi}^j}{v^2}\right)
\widetilde{\varphi}_{ij}\right]}-\frac{1}{n}$$ and
$\widetilde{\varphi}=\log \widetilde{u}$. Similar to what has been
done in the $C^1$ estimate (see Lemma \ref{Gradient}), we can deduce
a decay estimate of $\widetilde{u}(\cdot, s)$ as follows.

\begin{lemma} \label{lemma5-1}
Let $u$ be a solution of \eqref{Eq-}, then we have
\begin{equation}\label{Gra-est1}
|D\widetilde{u}(x, s)|\leq c_{12}\sup_{M^n}|D\widetilde{u}(\cdot,
0)|,
\end{equation}
where $c_{12}$ is a positive constant depending only on $c_{7}$,
$c_{8}$.
\end{lemma}

\begin{proof}
Set $\widetilde{\psi}=\frac{|D \widetilde{\varphi}|^2}{2}$. Similar
to that in Lemma \ref{Gradient}, we also can obtain
\begin{equation*}
\left\{
\begin{aligned}
&\frac{\partial \widetilde{\psi}}{\partial s}\leq
\widetilde{Q}^{ij}\widetilde{\psi}_{ij}+\widetilde{Q}^k\widetilde{\psi}_k
\quad &&\mathrm{in}~
M^n\times(0,\infty)\\
&D_\mu \widetilde{\psi} \leq 0 \quad &&\mathrm{on}~\partial M^n\times(0,\infty)\\
&\psi(\cdot,0)=\frac{|D\widetilde{\varphi}(\cdot,0)|^2}{2}
\quad&&\mathrm{in}~M^n.
\end{aligned}\right.
\end{equation*}
Using the maximum principle and Hopf's lemma, we can get the
gradient estimate of $\widetilde{\varphi}$, and then the conclusion
(\ref{Gra-est1}) follows.
\end{proof}

\begin{lemma}\label{rescaled flow}
Let $u$ be a solution of the flow \eqref{Eq-}. Then,
\begin{equation*}
\widetilde{u}(\cdot, s)
\end{equation*}
converges to a real number as $s\rightarrow +\infty$.
\end{lemma}
\begin{proof}
Set $P(t):= \mathcal{H}^n(M_{t}^{n})$, which, as before, represents
the $n$-dimensional Hausdorff measure of $M_{t}^{n}$ and is actually
the area of $M_{t}^{n}$. According to the first variation of a
submanifold, see e.g. \cite{ls},
 and the fact $\mbox{div}_{M_{t}^{n}}\nu=H$, we have
\begin{equation}\label{imcf-crf-for-01}
\begin{split}
P'(t)&=\int_{M_{t}^{n}} \mbox{div}_{M_{t}^{n}} \left( \frac{\nu}{f(|X|)H}\right) d\mathcal{H}^n\\
&=\int_{M_{t}^{n}} \sum_{i=1}^n \left\langle \nabla_{e_i}\left(\frac{\nu}{f(|X|)H}\right), e_i\right\rangle d\mathcal{H}^n\\
&=\int_{M_{t}^{n}}f^{-1}(u)d\mathcal{H}^n,
\end{split}
\end{equation}
where $\{e_{i}\}_{1\leq i\leq n}$ is some orthonormal basis of the
tangent bundle $TM_{t}^{n}$. By the condition of the function $f$ and \eqref{C^0}, we have
$$f^{-1}\left(e^{\mathcal{F}(t)+\varphi_2}\right)\leq f^{-1}(u)\leq f^{-1}\left(e^{\mathcal{F}(t)+\varphi_1}\right)$$
Hence, we have
\begin{equation*}
  f^{-1}\left(e^{\mathcal{F}(t)+\varphi_2}\right)P(t)\leq P'(t)\leq f^{-1}\left(e^{\mathcal{F}(t)+\varphi_1}\right)P(t).
\end{equation*}
Combining this fact with (\ref{imcf-crf-for-01}), the construction
process of function $\mathcal{F}$, and choosing the constant $c$ in
Lemma \ref{lemma3.1} as $\varphi_{1}$ and $\varphi_{2}$
respectively, we know
\begin{equation*}
\mathcal{H}^{n}(M^{n}_{0})\cdot \frac{e^{n\left(\mathcal{F}(t)+\varphi_2\right)}}{e^{n\varphi_2}}\leq P(t)\leq \mathcal{H}^{n}(M_0^n)\cdot \frac{e^{n\left(\mathcal{F}(t)+\varphi_1\right)}}{e^{n\varphi_1}}.
\end{equation*}
 Therefore, the rescaled
hypersurface $\widetilde{M}_s=M_{t}^{n} \Theta^{-1}$ satisfies the
following inequality
 \begin{eqnarray*}
  \frac{\mathcal{H}^{n}(M^{n}_{0})}{e^{n\varphi_2}}\leq
\mathcal{H}^n(\widetilde{M}_s)\leq
\frac{\mathcal{H}^{n}(M^{n}_{0})}{e^{n\varphi_1}},
 \end{eqnarray*}
 which implies that the area
of  $\widetilde{M}_s$ is bounded and the bounds are independent of
$s$. Here $\varphi_1=\inf_{M^n} \varphi(\cdot,0)$ and
$\varphi_2=\sup_{M^n} \varphi(\cdot,0)$. Together with
(\ref{imfcone-holder-01}), Lemma \ref{lemma5-1} and the
Arzel\`{a}-Ascoli theorem, we conclude that $\widetilde{u}(\cdot,s)$
must converge in $C^{\infty}(M^n)$ to a constant function
$r_{\infty}$ with
\begin{eqnarray*}
\frac{1}{e^{\varphi_{2}}}\left(\frac{\mathcal{H}^n(M_{0}^{n})}{\mathcal{H}^n(M^{n})}\right)^{\frac{1}{n}}\leq
r_{\infty}
\leq\frac{1}{e^{\varphi_{1}}}\left(\frac{\mathcal{H}^n(M_{0}^{n})}{\mathcal{H}^n(M^{n})}\right)^{\frac{1}{n}},
\end{eqnarray*}
i.e.,
\begin{eqnarray}\label{radius}
\frac{1}{\sup\limits_{M^{n}}u_{0}}\left(\frac{\mathcal{H}^n(M_{0}^{n})}{\mathcal{H}^n(M^{n})}\right)^{\frac{1}{n}}\leq
r_{\infty}
\leq\frac{1}{\inf\limits_{M^{n}}u_{0}}\left(\frac{\mathcal{H}^n(M_{0}^{n})}{\mathcal{H}^n(M^{n})}\right)^{\frac{1}{n}}.
\end{eqnarray}
This completes the proof.
\end{proof}

So, we have
\begin{theorem}\label{rescaled flow}
The rescaled flow
\begin{equation*}
\frac{d
\widetilde{X}}{ds}=\frac{f(\Theta)}{f(|X|)\widetilde{H}}\nu-\frac{1}{n}\widetilde{X}
\end{equation*}
exists for all time and the leaves converge in $C^{\infty}$ to a
piece of the round sphere of radius
$r_{\infty}$, where
$r_{\infty}$ satisfies (\ref{radius}).
\end{theorem}

\section*{Appendix}

\begin{proof}[Proof of Lemma \ref{EVQ}]
  Using the Gauss formula,
 we have
 \begin{equation*}
 \begin{split}
 \partial_{t}h_{ij}&=
 \partial_{t}\langle \partial_i \partial_j X, -\nu\rangle\\
 &=-\nabla^{2}_{ij} \Phi+\Phi h_{ik}h_{j}^{k}.
 \end{split}
 \end{equation*}
 Direct calculation results in
 \begin{equation*}
   \begin{aligned}
 \nabla^{2}_{ij}\Phi&=\Phi(-\frac{1}{H}H_{ij}+\frac{2H_i H_j}{H^2}) + 2f'f^{-1}\Phi u_{i}(\log H)_{j}\\
 &\quad
 -f'f^{-1}\Phi u_{ij} - (f''f^{-1} - 2f^{-2}(f')^{2})\Phi u_{i}u_{j}.
   \end{aligned}
 \end{equation*}
 Since we have
 \begin{eqnarray*}
 \Delta h_{ij}=H_{ij}+H h_{ik}h^{k}_{j}-h_{ij}|A|^2,
 \end{eqnarray*}
 so
 \begin{equation*}
   \begin{aligned}
 \nabla^{2}_{ij}\Phi&=-\Phi H^{-1}\Delta h_{ij}+\Phi h_{ik}h^{k}_{j}
 -\Phi H^{-1}|A|^2 h_{ij}+2\Phi\frac{H_i H_j}{H^2}\\
 &\quad
 +2f'f^{-1}\Phi u_{i}(\log H)_{j} - f'f^{-1}\Phi u_{ij} - (f''f^{-1} - 2f^{-2}(f')^{2})\Phi u_{i}u_{j}.
   \end{aligned}
 \end{equation*}
 Thus
 \begin{equation*}
   \begin{aligned}
 \partial_{t}h_{ij}-\Phi H^{-1}\Delta h_{ij}&=\Phi H^{-1}|A|^2 h_{ij}-\frac{2\Phi}{H^2}H_i H_j - 2f'f^{-1}\Phi u_{i}(\log H)_{j} \\
 &\quad
 + f'f^{-1}\Phi u_{ij} + (f''f^{-1} - 2f^{-2}(f')^{2})\Phi u_{i}u_{j}.
   \end{aligned}
 \end{equation*}
 Then
 \begin{equation*}
 \begin{split}
 \partial_t H&=  h_{ij}\partial_t g^{ij} + g^{ij} \partial_t h_{ij}\\
 &= \Phi H^{-1}\Delta H - \frac{2\Phi}{H^{2}}|\nabla H|^{2} - \Phi|A|^{2} - 2f'f^{-1}\Phi u_{i}(\log H)^{i} \\
 &\quad
 +f'f^{-1}\Phi\Delta u + (f''f^{-1} - 2f^{-2}(f')^{2})\Phi |\nabla u|^{2} \\
 &=
 f^{-1}H^{-2}\Delta H - 2f^{-1}H^{-3}|\nabla H|^{2} - f^{-1}H^{-1}|A|^{2} - 2f'f^{-2}H^{-2}u_{i}H^{i}\\
 &\quad
 + f'f^{-2}H^{-1}\nabla u + (f''f^{-1} - 2f^{-2}(f')^{2})f^{-1}H^{-1}|\nabla u|^{2}.
 \end{split}
 \end{equation*}
 Clearly
 \begin{eqnarray*}
 \partial_{t}w= \Phi + f'f^{-1}\Phi\nabla^{i}u\langle X, X_{i}\rangle + \Phi H^{-1}\nabla^{i}H\langle X, X_{i}\rangle,
 \end{eqnarray*}
 using the Weingarten equation, we have
 \begin{eqnarray*}
 w_i=h_{i}^{k}\langle X, X_k\rangle,
 \end{eqnarray*}
 \begin{eqnarray*}
 w_{ij}=h_{i,j}^{k}\langle X,
 X_k\rangle+h_{ij}-h_{i}^{k}h_{kj}\langle X, \nu\rangle
 =h_{ij, k}\langle X, X^k\rangle+h_{ij}-h_{i}^{k}h_{kj}\langle
 X, \nu\rangle.
 \end{eqnarray*}
 Thus
 \begin{eqnarray*}
 \Delta w= H+\nabla^i H\langle X, X_i\rangle - |A|^2 \langle X,
 \nu\rangle,
 \end{eqnarray*}
  and
 \begin{eqnarray*}
 \partial_t w= f^{-1}H^{-2} \Delta w + f^{-1}H^{-2} w |A|^2 + f'f^{-2}H^{-1}\nabla^{i} u\langle X, X_{i}\rangle.
 \end{eqnarray*}
 Hence
 \begin{equation*}
 \begin{split}
 \frac{\partial \Psi }{\partial t}&= - \frac{f'}{f^{2}Hw}\partial_{t}u -  \frac{1}{fH^2w} \partial_tH -\frac{1}{fHw^2} \partial_tw\\
 &= -f'f^{-3}H^{-2}w^{-2}u - (f''f^{-1} - 2f^{-2}(f')^{2})f^{-2}H^{-3}w^{-1}|\nabla u|^{2} + 2f^{-2}H^{-5}w^{-1}|\nabla H|^{2} \\
 &\quad
 + 2f'f^{-3}H^{-4}w^{-1}u_{i}H^{i} - f'f^{-3}H^{-3}w^{-1}\Delta u - f^{-2}H^{-4}w^{-1}\Delta H  - f^{-2}H^{-3}w^{-2}\Delta w \\
 &\quad
 -f'f^{-3}H^{-2}w^{-2}\nabla^{i}u\langle X, X_{i}\rangle.
 \end{split}
 \end{equation*}
 In order to prove (\ref{div-for-1}), we calculate
 $$\nabla_i \Psi =-f'f^{-2}H^{-1}w^{-1}\nabla_{i}u - f^{-1}H^{-2}w^{-1}\nabla_i H - f^{-1}H^{-1}w^{-2}\nabla_{i}w,$$
 and
 \begin{equation*}
 \begin{split}
 \nabla^2_{ij}\Psi&= -(f''f^{-1} - 2f^{-2}(f')^{2})f^{-1}H^{-1}w^{-1}\nabla_{i}u\nabla_{j}u + 2f^{-1}H^{-3}w^{-1}\nabla_{i}H\nabla_{j}H + 2f^{-1}H^{-1}w^{-3}\nabla_i w\nabla_j w \\
 & \quad
 + 2f'f^{-2}H^{-2}w^{-1}\nabla_i u\nabla_j H + 2f'f^{-2}H^{-1}w^{-2}\nabla_i u \nabla_j w + 2f^{-1}H^{-2}w^{-2}\nabla_i H\nabla_j w \\
 &\quad
 -f'f^{-2}H^{-1}w^{-1}\nabla_{ij}^{2}u - f^{-1}H^{-2}w^{-1}\nabla_{ij}^{2}H - f^{-1}H^{-1}w^{-2}\nabla_{ij}^{2}w.
 \end{split}
 \end{equation*}
 Thus
 \begin{equation*}
 \begin{split}
  f^{-1}H^{-2} \Delta \Psi
 &= - (f''f^{-1} - 2f^{-2}(f')^{2})f^{-2}H^{-3}w^{-1}|\nabla u|^{2} + 2f^{-2}H^{-5}w^{-1}|\nabla H|^{2} + 2f^{-2}H^{-3}w^{-3}|\nabla w|^{2}  \\
 &\quad
 + 2f'f^{-3}H^{-4}w^{-1}\nabla_i u\nabla^{i}H + 2f'f^{-3}H^{-3}w^{-2}\nabla_i u\nabla^{i}w + 2f^{-2}H^{-4}w^{-2}\nabla_i H\nabla^{i}w \\
 &\quad
 - f'f^{-3}H^{-3}w^{-1}\Delta u - f^{-2}H^{-4}w^{-1}\Delta H - f^{-2}H^{-3}w^{-2}\Delta w.
 \end{split}
 \end{equation*}
 So we have
 \begin{equation*}
 \begin{split}
 &\mbox{div} (f^{-1} H^{-2} \nabla \Psi)=- f'f^{-2}H^{-2}\nabla_i \Psi\nabla^{i}u - 2f^{-1}H^{-3}\nabla_i\Psi\nabla^{i}H + f^{-1}H^{-2}\Delta\Psi\\
 &=-(f''f^{-1} - 3f^{-2}(f')^{2})f^{-2}H^{-3}w^{-1}|\nabla u|^{2} + 4f^{-2}H^{-5}w^{-1}|\nabla H|^{2} + 2f^{-2}H^{-3}w^{-3}|\nabla w|^{2}\\
 &\quad + 5f'f^{-3}H^{-4}w^{-1}\nabla_{i}u\nabla^{i}H + 3f'f^{-3}H^{-3}w^{-2}\nabla_{i}u\nabla^{i}w + 4f^{-2}H^{-4}w^{-2}\nabla_{i}H\nabla^{i}w \\
 & \quad -f'f^{-3}H^{-3}w^{-1}\Delta u - f^{-2}H^{-4}w^{-1}\Delta H - f^{-2}H^{-3}w^{-2}\Delta w,
 \end{split}
 \end{equation*}
 and
 \begin{equation*}
   \begin{split}
 2H^{-1} w |\nabla \Psi|^2&= 2(f')^{2}f^{-4}H^{-3}w^{-1}|\nabla u|^{2} + 2f^{-2}H^{-5}w^{-1}|\nabla H|^{2} + 2f^{-2}H^{-3}w^{-3}|\nabla w|^{2} \\
 &\quad
 +4f'f^{-3}H^{-4}w^{-1}\nabla_i u\nabla^{i}H + 4f'f^{-3}H^{-3}w^{-2}\nabla_i u\nabla^{i}w + 4f^{-2}H^{-4}w^{-2}\nabla_i H\nabla^{i}w.
   \end{split}
 \end{equation*}
 As above, we have
 \begin{equation*}
   \begin{split}
 &\quad\frac{\partial \Psi }{\partial t}-\mbox{div} ( f^{-1}H^{-2} \nabla
 \Psi)+2H^{-1} w |\nabla \Psi|^2 \\
 &=
 -f'f^{-3}H^{-2}w^{-2}u - f'f^{-3}H^{-2}w^{-2}\nabla^{i}u\langle X, X_{i}\rangle + (f')^{2}f^{-4}H^{-3}w^{-1}|\nabla u|^{2} \\
 &\quad
 + f'f^{-3}H^{-4}w^{-1}\nabla_i u\nabla^{i}H + f'f^{-3}H^{-3}w^{-2}\nabla_i u\nabla^{i}w \\
 & =
 - f'\frac{u}{f}\Psi^{2} - f'f^{-1}\Psi^{2}\nabla^{i}u\langle X, X_{i}\rangle - f'f^{-2}H^{-2}\nabla_i u\nabla^{i}\Psi.
   \end{split}
 \end{equation*}
 The proof is finished.
 \end{proof}

\vspace {5 mm}

\section*{Acknowledgments}

This research was supported in part by the NSF of China (Grant No.
11926352), the Fok Ying-Tung Education Foundation (China), and Hubei
Key Laboratory of Applied Mathematics (Hubei University).

\end{document}